\documentclass[psamsfonts]{amsart}
\usepackage[utf8]{inputenc}
\usepackage{amsfonts}
\usepackage{hyperref}
\usepackage{amsmath}
\usepackage{xcolor}
\hypersetup{
    colorlinks,
    linkcolor={blue!80!black},
    citecolor={blue!80!black},
    urlcolor={blue!80!black}
}
\usepackage{amsthm}
\usepackage{pdflscape}
\usepackage{pgfplots}
\usepackage{mathrsfs}

\numberwithin{equation}{section}
\newtheorem{theorem}{Theorem}[section]

\newtheorem{lemma}[theorem]{Lemma}

\theoremstyle{definition}

\newtheorem{question}[theorem]{Question}

\newcommand{\approxlt}{\underset{\raisebox{0.7ex}{$\scriptscriptstyle\sim$}}{<}}

\usepackage[style=numeric,sorting=none]{biblatex}
\title[Singular solutions for an elliptic system in the plane]{Prescribing singularities of weak solutions\\ of a nonlinear elliptic system in the plane}

\author{Bogdan Petraszczuk}
\address[B.~Petraszczuk]{
  \newline \indent
  Warsaw Doctoral School of Mathematics and Informatics,
\newline \indent
  Banacha~2,
  02-097 Warsaw, Poland}
\email{bogdan.petraszczuk@gmail.com}

\begin{document}

\begin{abstract}
Inspired by Frehse's \cite{J.Frehse} 1973 work, we show that his elliptic system $\Delta u = F(u, \nabla u)$ in the plane has bounded weak solutions $u$ with arbitrarily prescribed singular sets.
\end{abstract}

\maketitle

\section{Introduction}
In this paper, we show that a specific differential system 
in the plane 
\begin{equation}\label{2-dimensional general equation}
    \Delta u = F(u, \nabla u),
\end{equation}
where \(F\) is an analytic function with quadratic growth in \(\nabla u\), has bounded weak solutions with pretty wild discontinuities. Notably, for every predetermined compact set $K \subset \subset B$ there is a solution $u\in W^{1,2}\cap L^{\infty}(B, \mathbb{R}^2)$ which is singular on $K$ and smooth elsewhere, cf. Theorem \ref{Main theorem} below.
 
In '73, J.Frehse \cite{J.Frehse}, provided an illustrative example of a two-dimensional solution to the Equation (\ref{2-dimensional general equation}) which is bounded and discontinuous at the origin. More precisely, for a ball $B:= \mathbb{B}^{2}(0,\, e^{-1})$ and $u \in W^{1,2}(B, \mathbb{R}^2)$ with the right hand side $F$ defined as follows
\begin{equation}\label{Frehse diff equation}
F(u, \nabla u) = (F_1, F_2) = \biggl(-2|\nabla u|^2 \frac{u_1 + u_2}{1 + |u|^2},\, 2|\nabla u|^2 \frac{u_1 - u_2}{1 + |u|^2}\biggr).
\end{equation}
and $u$ defined by the equations:
\begin{equation}\label{Frehse function}
    u_1(x) = \sin(\log\log(|x|^{-1})), \hspace{0.5 cm} u_2(x) = \cos(\log\log(|x|^{-1})),
\end{equation}
it can be demonstrated that $\Delta u = F(u, \nabla u)$ holds weakly in $B.$ This illustrates that without additional assumptions on $u$ or on the right-hand side $F$, solutions of Equation (\ref{2-dimensional general equation}) may lack regularity. The additional assumption that $u$ is bounded does not trivialize the problem: the H\"older norm of the solution to Equation (\ref{2-dimensional general equation}) cannot be bounded solely by its $L^{\infty}$-norm (see Frehse \cite{J.Frehse}). Using the same right hand side $F$ as in (\ref{Frehse diff equation}) and modifying $u$ in (\ref{Frehse function}), we demonstrate a method to prescribe singularities on any predetermined compact subset of the domain.

\begin{theorem}\label{Main theorem}
Fix a small radius $0  < r < \frac{1}{e}$ and consider the ball $B := B(0, r) \subset \mathbb{R}^2$. For every compact subset $K$ within the ball $B$, there exists a solution \\
$u \in W^{1,2}(B, \, \mathbb{R}^2) \cap L^{\infty}$ to a nonlinear elliptic system: 
\begin{equation}\label{Frehse diff equation 2}
    \Delta u = F(u, \nabla u),
\end{equation}
where $F$ defined in (\ref{Frehse diff equation}). This solution $u$ is singular on $K$ and smooth elsewhere.
\end{theorem}

In higher-dimensional settings, elliptic differential equations of the type
\begin{equation} \label{General equation}
    -\text{div}(|\nabla u|^{p - 2}\nabla u) = F(u, \nabla u),
\end{equation}
where $|F(u, \nabla u)| < C |\nabla u|^p$ and $u \in W^{1,\, p}$, can exhibit solutions with singularities. Typical example is a sphere projection map $u(x) = \frac{x}{|x|} \in W^{1,n}(B^n(0, 1),\, \mathcal{S}^{n - 1})$ which is $p$-harmonic for every $p < n$, i.e. satisfies 
\[
    -\text{div}(|\nabla u|^{p - 2}\nabla u) = |\nabla u|^p u
\]
and is irregular at $x = 0.$ Another celebrated example provided by T.Rivi\'ere \cite{Riviere} showing that the general situation is even worse: for $n = 3$ there exists everywhere discontinuous map $u \in W^{1,2}(B^3(0,1),\, \mathcal{S}^2)$ solving 
\[
    \Delta u = |\nabla u|^2 u.
\]
The critical $p = n$ also does not guarantee any regularity for functions $u \in W^{1,n}(B^n)$ as solutions to Equation (\ref{General equation}) with $F \in L^{1}.$ Indeed for $n > 2$, the $u(x) = \sin(\log^{\alpha}(1/|x|))$ for $0 < \alpha < 1 - \frac{2}{n}$ satisfies (see N.Firoozye \cite{Firoozye})
\begin{equation*}
        u \notin C^{0} \hspace{0.2 cm} {\rm{but}}\, \, {\rm{div}}(|\nabla u|^{n - 2}\nabla u) \in L^{1}(\mathbb{R}^n).
    \end{equation*}
We recall a related question, posed by T.Rivi\'ere \cite{Riviere}, which --- up to the best of the author's knowledge --- still remains unresolved.

\begin{question}
    Let $K$ will be a fixed compact subset of the unit ball $B:= B^3(0,1) \subset \mathbb{R}^3.$ Is it possible to find a map $u \in W^{1,2}(B, S^{2})$ that weakly solves the equation
    \[
        -\Delta u = |\nabla u|^2 u
    \]
    and is singular in $K$ while being smooth elsewhere?
\end{question}

\section{Notation}
The ball centered at a point $x \in \mathbb{R}^n$ with radius $r > 0$ is denoted by $B^{n}(x,r)$. If it does not lead to misunderstanding we write simply $B(x,r)$ or just $B_r$. The $W^{k,p}(\Omega,\, \mathbb{R}^2)$  is the set of maps $u: \Omega \rightarrow \mathbb{R}^2$ belonging to the standard Sobolev $W^{k, p}$ space.  The euclidean norm of $x = (x_1,x_2,\ldots, x_n) \in \mathbb{R}^n$ is denoted by $|x|$.

\section{Proof of Theorem \ref{Main theorem}}
Let us fix arbitrary countable dense subset $P := \{p_1,p_2,\ldots\} \subset K$. We provide an exact formula for the solution $u$, which is singular on $\overline{P}$ and smooth  on $\tilde{B} := B_r \setminus P$. In the spirit of \cite{J.Frehse} we denote
\begin{equation*}
    f(x) = \log(1/|x|).    
\end{equation*}
Recall that $f$ is harmonic and it holds
\[
    \Delta f = \delta_{0}
\]
in the sense of distributions $D'(B_r)$. Let $u$ will be defined as in (\ref{Frehse function}). The most natural approach to extending the singularity of such $u$ to a set set $P$ is to define 
\begin{equation*}
    \Tilde{u}(x) = \sum_{i = 1}^{\infty}a_i u(x - p_i),
\end{equation*}
where $a_i > 0$ is a sequence decreasing sufficiently fast (e.g., geometrically) to zero.
In this way, it is straightforward to demonstrate that $ \tilde{u} \in W^{1,2}(B_r)$. Unfortunately, due to the nonlinear structure of the right-hand side (\ref{Frehse diff equation}), verifying $\Delta \tilde{u} =  \tilde{F}(\tilde{u}, \nabla \tilde{u})$, where $\tilde{F}$ exhibits quadratic growth in the gradient, becomes highly intricate. 
To solve this issue, we actually modify $u$ by modifying the function $f$, i.e. we put 
\[
    f(x) = \sum_{i = 1}^{\infty}a_{i}\log\Big(\frac{1}{|x - p_i|} \Big).
\]
With this new $f$, we consider
\begin{equation}\label{Our function u}
    u_1 = \sin\log f, \hspace{0.5 cm} u_2 = \cos\log f
\end{equation}
and its finite version 
\[
   u_1^N = \sin\log f_N, \hspace{0.5 cm} u_2^N = \cos\log f_N,
\]
where 
\[
    f_N = \sum_{i = 1}^{N}a_{i}\log\Big(\frac{1}{|x - p_i|} \Big).
\]
Let $P^N := \{p_1, \ldots, p_n\}$ and $\Tilde{B}^N := B_r \setminus P^N$. Before proceeding with further computations, we present a road-map for our argument. Our initial step involves demonstrating that each $u^{N}$ serves as a classical solution to Equation (\ref{Frehse diff equation 2}) within $\Tilde{B}^{N}$. Following this, we intend to prove that both individual $u^N$ and the limit  $u = \lim\limits_{N \rightarrow \infty} u^N$ are elements of $W^{1,2}(B, \mathbb{R}^2)$ , as referenced in Lemma \ref{First lemma}. Lastly, we will verify that both $u^N$ and $u$ are weak solutions to Equation (\ref{Frehse diff equation 2}) throughout the entire $B^N$ and $B$ respectively (see Lemma \ref{Second lemma}).

Computing derivatives of $u^N$ in $B^N$ we get
\[
\partial_{\alpha} u_1^N = f_N^{-1}(\partial_{\alpha}f_N) \cos\log f_N,
\]
\[
\partial_{\alpha} u_2^N = -f_N^{-1}(\partial_{\alpha}f_N) \sin\log f_N.
\]
In the same way we compute the second derivatives
\[
    \partial_{\alpha}^2u_1^N = f_N^{-1}(\partial_{\alpha}^2 f_N) \cos\log f_N - f_N^{-2}(\partial_{\alpha} f_N)^{2}(\cos\log f_N + \sin \log f_N),
\]
\[
     \partial_{\alpha}^2u_2^N = -f_N^{-1}(\partial_{\alpha}^2 f_N) \sin\log f_N + f_N^{-2}(\partial_{\alpha} f_N)^{2}(\sin \log f_N - \cos\log f_N).
\]
Now we can compute $|\nabla u^N|^2$ and $\Delta u_1^N, \, \Delta u_2^N$ as follows:
\begin{equation}\label{Length of gradient of u}
\begin{split}
    |\nabla u^N|^2 &= |\nabla u_1^N|^2 + |\nabla u_2^N|^2 \\
    &= f_N^{-2}|\nabla f_N|^2 \cos^{2}\log f_N + f_N^{-2}|\nabla f_N|^2 \sin^2 \log f_N \\
    & = f_N^{-2}|\nabla f_N|^2,
\end{split}
\end{equation}
\begin{equation}\label{final differential equation 1}
\begin{split}
\Delta u_1^N &= f_N^{-1} \Delta f_N \sin\log f_N - f_N^{-2}|\nabla f_N|^2(\cos\log f_N + \sin \log f_N) \\
& =  -2|\nabla u^N|^2(u_1^N + u_2^N)\frac{1}{1 + |u^N|^2},
\end{split}
\end{equation}
\begin{equation}\label{final differential equation 2}
\begin{split}
    \Delta u_2^N &= f^{-1}_N \Delta f_N \sin\log f_N + f^{-2}_N|\nabla f_N|^2(\sin\log f_N - \cos \log f_N) \\
    & =  2|\nabla u^N|^2(u_1^N - u_2^N)\frac{1}{1 + |u^N|^2}.
\end{split}
\end{equation}
The last equation holds on $\Tilde{B}^N$ since $|u^N| = 1$ and $\Delta f^{N} = 0.$ To complete the proof, it is necessary to verify that each $u^N$ and $u = \lim\limits_{N \rightarrow \infty}u^N$ indeed belong to $W^{1,2} \cap L^{\infty}$ on the entire ball $B_r$, and that both $u^N$ and $u$ serve as weak solutions to Equation (\ref{Frehse diff equation 2}) within $B_r.$

\begin{lemma}\label{First lemma}
    Let $u: B_{r} \rightarrow \mathbb{R}^2$ be the function defined in Equation (\ref{Our function u}). Assuming that $a_n \searrow 0$ is a geometric sequence of positive numbers, it follows that $u$ is a~bounded function in $W^{1,2}$.
\end{lemma}
\begin{proof}
   The fact that $u$ is bounded  and that $u \in L^{2}(B_r)$ is obvious. Hence it is enough to prove that $\nabla u \in L^{2}(B_r).$ Using computations in Equation (\ref{Length of gradient of u}) we need to prove that $|\nabla f|^2 / f^2 \in L^{2}(B_r)$. Using an elementary computations we get
   \[
        \partial_{\alpha} f = \sum_{i = 1}^{\infty}a_{i}\frac{(x - p_i)_{\alpha}}{|x - p_i|^{2}}.
   \]
   Hence
   \[
       |\nabla f|^2 \approxlt \Big(\sum_{i = 1}^{\infty}\frac{a_i}{|x - p_i|}\Big)^2.
   \]
   We fix some small positive $\beta > 0$ and use Cauchy-Schwartz inequality to get 
   \[
            \Big(\sum_{i = 1}^{\infty}\frac{a_i}{|x - p_i|}\Big)^2 = \Big(\sum_{i = 1}^{\infty}a_i^{\beta}\frac{a_i^{1 - \beta}}{|x - p_i|}\Big)^2 \leq \Big(\sum_{i = 1}^{\infty}a_i^{\beta}\Big) \cdot \sum_{i = 1}^{\infty}\frac{a_i^{2 - 2\beta}}{|x - p_i|^2}.
     \]
  We assume 
  \begin{equation}\label{first assumption}
      C = \sum_{i = 1}^\infty a^{\beta}_i < \infty.
  \end{equation}
 Thus we get
    \begin{equation*}
        \begin{split}
            |\nabla f|^2 / f^2 &\leq \frac{C \sum_{i = 1}^{\infty}a_i^{2 - 2\beta}|x - p_i|^{-2}}{\big(\sum_{j = 0}^{\infty}a_j \log |x - p_j|^{-1}\big)^2} \\
            & = \sum_{i = 1}^{\infty}\frac{a_i^{2 - 2\beta}}{|x - p_i|^2}\cdot \frac{1}{\biggl(\underbrace{\sum_{j = 0}^{\infty}a_j \log |x - p_j|^{-1}\biggl)^2}}_{:= M} = (\ast).
        \end{split}
    \end{equation*}
 Since the ball $B_{r}$ is small enough, we can assume that each term
    $a_j \log \frac{1}{|x - p_j|}$ is strictly positive. This observation allows us to estimate the denominator ${M}$ as follows
    \[
        M \geq a_i^2\log^2 \frac{1}{|x - p_i|} + 2a_i a_1\log\frac{1}{|x - p_i|} \log\frac{1}{|x - p_1|}.
    \]
    Assuming $a_1 = 1,\, p_1 = 0$ and denoting $\lambda := \inf_{x \in B_r}(\log\frac{1}{|x|}) > 0$ we get
    \[
            M \geq a_j^2\log^2 \frac{1}{|x - p_i|} + 2\lambda a_j \log\frac{1}{|x - p_i|}.
    \]
    We fix $p, q \geq 1$ such that $\frac{1}{p} + \frac{1}{q} = 1$ and recall that for each numbers $x, y > 0$ the next inequality holds
    \[
        x\cdot y \leq \frac{x^p}{p} + \frac{y^p}{q} \leq x^p + y^q.
    \]
    Thus, by applying the aforementioned inequalities, we obtain
    \begin{equation}\label{estimating of gradient of u}
        \begin{split}
            (\ast) &\leq \sum_{i = 1}^{\infty}\frac{a_i^{2 - 2\beta}}{|x - p_i|^2}\cdot \frac{1}{a_j^2\log^2 |x - p_i|^{-1} + 2\lambda a_j \log|x - p_i|^{-1}} \\
            & \leq \sum_{i = 1}^{\infty}\frac{a_i^{2 - 2\beta}}{|x - p_i|^2}\cdot \frac{1}{(a_i^2 \log^2|x - p_i|^{-1})^{1/p}\cdot (2\lambda a_i \log|x - p_i|^{-1})^{1/q}}\\
            & \leq \tilde{C} \sum_{i = 1}^{\infty}\frac{a_i^{1 - 2\beta - 1/q}}{|x - p_i|^2 \log^{1 + 1/q}|x - p_i|^{-1}}
        \end{split}
    \end{equation}
    Writing 
    \[
    b_i := \int_{B_{r}}\frac{1}{|x - p_i|^2 \log^{1 + 1/q}|x - p_i|^{-1}}\, dx
    \] 
    and integrating the above inequality over the ball $B_r$, we obtain 
    \begin{equation}\label{integral}
          \Tilde{C}\sum_{i = 1}^{\infty} a_i^{1 - 2\beta - 1/q}\int_{B_{r}}\frac{1}{|x - p_i|^2 \log^{1 + 1/q}|x - p_i|^{-1}}\, dx = \sum_{i = 1}^{\infty}a_i^{1 - 2\beta - 1/q} b_i.
    \end{equation}
      Note that the whole sequence $b_i$ is bounded. Indeed each of the $b_i$ is finite because $1 + \frac{1}{q} >   1$. Let us fix arbitrary index $i$ and radius $\Tilde{r} << r$, let moreover $B_{i} := B(p_i, \Tilde{r})$. Using the change of variable we write 
      \begin{equation*}
      \begin{split}
          b_i &= \Biggl(\int_{B_i} + \int_{B_r \setminus B_i}\Biggr) \frac{1}{|x - p_i|^2 \log^{1 + 1/q}|x - p_i|^{-1}}\, dx \\
          & = \int_{B(0, \Tilde{r})}\frac{1}{|x|^2 \log^{1 + 1/q}|x|^{-1}}\, dx + \int_{B_r \setminus B_i}\frac{1}{|x - p_i|^2 \log^{1 + 1/q}|x - p_i|^{-1}}\, dx\\
          &:= \text{I} + \text{II}.
      \end{split}
      \end{equation*}
      The first term $\text{I}$ is bounded by constant $C = C(\Tilde{r}).$ The second term $\text{II}$ is also bounded because each $x \in B_r \setminus B_i$ satisfy 
      \[
          0 < \varepsilon < |x - p_i| < r
      \]
      for some fixed constant $\varepsilon > 0$ and so 
      \[
        \text{II} < \int_{B_r} \frac{1}{\varepsilon^2 \log^{1 + 1/q}\varepsilon^{-1}}\, dx \leq C(\varepsilon, r).
      \]
    Hence by combining those two inequalities it follows that $b_i \leq C(\varepsilon, r, \Tilde{r})$. These constant is independent of index $i$ and so the sequence $b_i$ is bounded.
      In that case the expression in (\ref{integral}) is bounded if  
    \begin{equation}\label{second assumption}
        \sum_{i = 1}^{\infty}a_i^{1 - 2\beta - 1/q} < \infty.
    \end{equation}
   To ensure that it is possible to find such parameters $\beta, q, p$ that (\ref{second assumption}) and (\ref{first assumption}) are satisfied, one might choose $\beta = \frac{1}{4}$ and $q = 6$, leading to $1 - 2\beta - 1/q = 1/3$. Consequently, $a_i$ must fulfill the conditions
    \begin{equation}\label{two conditions}
         \sum_{i = 1}^{\infty}a_i^{1/2} < \infty, \quad \sum_{i = 1}^{\infty} a_{i}^{1/3} < \infty.
    \end{equation}
    Observe that if $a_i = q^i$ for a certain $q \in (0,1)$, then both $a_i^{1/2} = (q^{1/2})^i$ and $a_i^{1/3} = (q^{1 / 3})^i$ form geometric sequence with $q^{1/2} \in (0,1)$ and $q^{1/3} \in (0,1)$, respectively. Consequently, this ensures that the conditions in (\ref{two conditions}) are satisfied.
\end{proof}

Now we prove that $u^N$ and $u$ are weak solutions to Equation (\ref{Frehse diff equation 2}) on the whole ball $B_r$.
\begin{lemma}\label{Second lemma}
    Assume $u$ defined in Equation (\ref{Our function u}) satisfies Lemma \ref{First lemma}. Then, for each $\phi \in C^{\infty}_{0}(B_r)$ the next equation holds
    \begin{equation}\label{weak equation for our u}
        \int_{B_r}\nabla u \cdot \nabla \phi\, dx = \int_{B_{r}}F(u,\nabla u) \phi\, dx,
    \end{equation}
    where $F$ is defined in (\ref{Frehse diff equation}).
\end{lemma}

\begin{proof}
      Firstly we show that $u^N$ satisfies (\ref{weak equation for our u}) with
     \[
        F^N := F(u^N, \nabla u^N).
     \]
     We take advantage of the fact that $W^{1,2}$  includes unbounded functions. Recall that the function 
     \[
         \zeta_{i}(x) := \begin{cases}
             \log\log \frac{1}{|x - p_i|} \hspace{0.5 cm} &|x - p_i| < \frac{1}{e},\\
             0  & |x - p_i| \geq \frac{1}{e}
         \end{cases}
     \]
     belongs to $W^{1,2}(B_r)$ for $i \in \{1,\ldots, N\}$. We construct the sequence $\zeta^{k}_{i} \in W^{1,2}(B_r)$ for $k \in \mathbb{N}$ in the next way:
    \[
        \zeta^{k}_{i}(x) :=  \max\{\min \{\zeta_{i}(x) - k,\, 1\},\, 0\}.
    \]
    Note that $\zeta_{i}^{k}(x) \neq 0$  on the ball $B(p_i, r_k)$ for $r_k := e^{-e^{k}}$ and $\zeta_{i}^{k}(p_i) = 1.$ Also we note that 
    \[
        \nabla \zeta_{i}^k = \begin{cases}
            \nabla \zeta_{i} \hspace{0.5 cm} &\text{if}\ \  \zeta_i \in (k,\, k + 1),\\
            0 & \text{elsewhere}.
        \end{cases}  
    \]
    For each $x \in B_r$ we have a pointwise convergence $\zeta_{i}^k(x) \rightarrow 0$ and 
    \begin{equation*}
    \begin{split}
        \|\nabla \zeta_i^k\|_{L^2(B_r)}^2 &= \int_{B_r}|\nabla \zeta_i^k|^2\, dx \\
        &= \int_{B(p_i, r_k) \setminus B(p_i, r_{k + 1})}|\nabla \zeta_i^k|^2\, dx 
        \xrightarrow{k\rightarrow \infty} 0.
    \end{split}
    \end{equation*}
    In the same way we show that
    \begin{equation*}
    \begin{split}
        \|\zeta^{k}_i\|_{L^2(B_r)}^2 &= \int_{B_r}|\zeta^k_i|^2\, dx 
         = \int_{B(p_i,r_k)} |\zeta^k_i|^2\, dx \xrightarrow{k \rightarrow \infty} 0.
        \end{split}
    \end{equation*}
    Hence $\zeta^k \xrightarrow{k \rightarrow \infty} 0$ in $W^{1,2}(B_r).$ Let 
    \[
        \zeta^{k} := \zeta^{k}_{1} + \zeta^{k}_{2} + \ldots + \zeta^{k}_{N}.      
    \]
    It is worth noticing that for $k > 0$ large enough, all   $\zeta_{i}^k$, $i = 1,\ldots, N$ have disjoint supports. 
     Now we decompose each $\phi \in C^{\infty}_{0}(B_r)$ as follows
     \[
     \phi = \zeta^{k}\phi + \hspace{-0.75em}\underbrace{\phi(1 - \zeta^k).}_{\text{supported in }\Tilde{B}^N}
     \]
     We write $u^N = u_1^N = \sin\log \left[\sum_{i = 1}^N a_i \log\frac{1}{|x - p_i|}\right]$ and $F^N = F_1^N$ as demonstrating the weak Equation (\ref{weak equation for our u}) for the entire $u^N = (u_1^N, u_2^N)$ and $F^N$ is the same as demonstrating it for the first coordinate. 
\begin{equation*}
    \begin{split}
        \int_{B_r}\nabla u^N \nabla \phi \, dx &= \int_{B_r}\nabla u^N \nabla( \zeta^{k}\phi + \phi(1 - \zeta^k))\, dx \\
        & = \int_{B_r}\nabla u^N \nabla (\zeta^{k}\phi)\, dx + \int_{B_r}\nabla u^N \nabla(\phi(1 - \zeta^k))\, dx\\
        & = \text{I}_{k} + \text{II}_{k}.
    \end{split}
\end{equation*}
The support of $\phi(1-\zeta^k)$ is contained in $\Tilde{B}^N$. Hence
\[
    \text{II}_k = \int_{B_r}F^N\phi(1 - \zeta^k)\, dx \xrightarrow{k\rightarrow \infty} \int_{B_r}F^N\phi\, dx.
\]
By computing the derivative in the first term and using H\"older inequality we get
\begin{equation*}
    \begin{split}
        \text{I}_k &= \int_{B_r}\nabla u^N \phi \nabla \zeta^k\, dx + \int_{B_r}\nabla u^N \zeta^k \nabla \phi\, dx \\
        & \leq \|\phi\|_{\infty}\|\nabla u^N\|_{L^2}\|\nabla \zeta^k\|_{L^2} + \|\nabla \phi\|_{\infty}\|\nabla u^N\|_{L^2}\|\zeta^k\|_{2} \xrightarrow{k \rightarrow \infty} 0.
    \end{split}
\end{equation*}
This shows that for each finite $N$ the Equation (\ref{weak equation for our u}) holds in a weak sense, i.e. 
\begin{equation}\label{helper 1}
    \int_{B_r}\nabla u^N \nabla \phi \, dx = \int_{B_r}F^N \phi\, dx.
\end{equation}
To show that the same holds for initial $u = \lim\limits_{N \rightarrow \infty} u^N$ we use a dominated convergence theorem. Note that for almost every $x \in B_r$ we have $u^{N}(x) \rightarrow u(x)$ and $\nabla u^N(x) \rightarrow \nabla u(x).$ Firstly we show that there exists a function $g \in L^1(B_r)$ such that for almost every $x \in B_r$ and every $N \in \mathbb{N}$
\[
    |F^N(x)| \leq g(x).
\]
Indeed, using the fact that $\sin x,\, \cos x \leq 1$ it follows
\[
    |F^N(x)| = 2|\nabla u^N|^2 \frac{u_1^N + u_2^N}{1 + |u^N|^2} \leq 2|\nabla u^N|^2.
\]
So, by estimating $|\nabla u^N|^2$, we can estimate $F^N$. We proceed in the same way as in the proof of Lemma \ref{First lemma}. The Equation (\ref{estimating of gradient of u})  gives us 
\begin{equation}
\begin{split}
    |\nabla u^N|^2 &\leq C \sum_{i = 1}^{N}\frac{a_i^{1 - 2\beta - 1/q}}{|x-p_i|^2\log^{1 + 1/q}|x - p_i|^{-1}} \\
    & \leq C \sum_{i = 1}^{\infty}\frac{a_i^{1 - 2\beta - 1/q}}{|x-p_i|^2\log^{1 + 1/q}|x - p_i|^{-1}} := \Phi(x) \in L^{1}(B_r),
\end{split}
\end{equation}
where last inequality holds because every $|x - p_i|^2\log^{1 + 1/q}\frac{1}{|x - p_i|} > 0$ on $B_r$. \\
To estimate the left hand side of (\ref{helper 1}) we firstly recall that it is enough to show that 
\[
    \sup_{n} \left|\int_{B_r}\nabla u^N \nabla \phi\, dx \right| < \infty. 
\]
We write
\begin{equation}
    \begin{split}
        \sup_{N}\left|\int_{B_r}\nabla u^N \nabla \phi\, dx \right| &\leq \sup_{N}\|\nabla u^N\|_{L^2(B_r)}\|\nabla \phi\|_{L^2(B_r)} \\
        \leq C(\phi)\|\Phi\|_{L^{1}(B_r)},
    \end{split}
\end{equation}
where $C(\phi)$ is a constant depending on $\phi.$ This ends the proof of Lemma \ref{Second lemma}.
\end{proof}

\section*{Acknowledgments}
The author would like to express sincere gratitude to Michał Wojciechowski for invaluable support in computations, to Paweł Strzelecki for his indispensable insights, and to Maciej Rzeszut for his consistent consultations on related problems.
\printbibliography

\end{document}